\RequirePackage{fix-cm}
\documentclass[smallcondensed]{svjour3}     
\smartqed  
\usepackage{graphicx}
\usepackage{amsmath}%
\usepackage{amsfonts}%
\usepackage{amssymb}%
\usepackage{mathrsfs}
\usepackage[matrix,arrow,curve]{xy}
\usepackage{pgf,tikz}
\usepackage{mathrsfs}
\usetikzlibrary{arrows}
\usetikzlibrary[patterns]
\begin{document}

\title{On a topology and limits for inductive systems of $C^*$-algebras over partially ordered sets
}
\subtitle{}

\titlerunning{On a topology and limits for inductive systems of $C^*$-algebras}        

\author{Gumerov R.N., Lipacheva E.V., Grigoryan T.A.}


\institute{R. Gumerov \at
              Chair of Mathematical Analysis,  N.I. Lobachevsky Institute of Mathematics and Mechanics,  Kazan (Volga Region) Federal University, Kremlevskaya 35, Kazan, Russian Federation, 420008 \\
              \email{Renat.Gumerov@kpfu.ru}           
           \and
           E. Lipacheva \at
              Chair of Higher Mathematics, Kazan State Power Engineering University, Krasnoselskaya 51,  Kazan, Russian Federation, 420066\\
              Tel.: +7(843)519-42-84\\
              \email{elipacheva@gmail.com}
              \and
           T. Grigoryan \at
              Chair of Higher Mathematics, Kazan State Power Engineering University, Krasnoselskaya 51,  Kazan, Russian Federation, 420066\\
              \email{tkhorkova@gmail.com}
}

\date{Received: date / Accepted: date}

\maketitle

\begin{abstract}
Motivated by algebraic quantum field theory and our previous work we study properties of inductive systems of \ $C^*$-algebras over arbitrary partially ordered sets. A partially ordered set can be
represented as the union of  the  family of its maximal upward directed subsets indexed by elements of a certain set. We
consider a topology on the set of indices generated by a base of neighbourhoods. Examples of those topologies with different properties are given. An inductive system of $C^*$-algebras and its inductive limit arise naturally over each maximal upward directed subset. Using those inductive limits, we construct different types of $C^*$-algebras. In particular, for  neighbourhoods of the  topology on the set of indices we deal with the $C^*$-algebras which are the direct products of those inductive limits.
The present paper is concerned with  the above-mentioned topology and the algebras arising from an inductive system of $C^*$-algebras over a partially ordered set. We show that there exists a connection between properties of that topology and those $C^*$-algebras.

\end{abstract}

\section{Introduction}
\label{intro}
The motivation for the present paper comes from
algebraic quantum field theory \cite{HaagKastler1964}--\cite{Vasselli2012} 
and our previous work on inductive systems of $C^*$-algebras  \cite{GLS2016}--\cite{GLS2018}.

The general framework of algebraic quantum field theory is given
by a covariant functor. Usually that functor acts from a category
whose objects are topological spaces with additional structures and its
morphisms are structure preserving embeddings into a category  describing the
algebraic structure of observables.  The standard assumption in quantum physics
is that the second category consists of unital $C^*$-algebras and
unital embeddings of $C^*$-algebras.
The basic tool of the algebraic approach to quantum fields over a spacetime is  a net of $C^*$-algebras defined over a partially ordered set defined as a
suitable set of regions of the spacetime ordered under inclusion \cite{HaagKastler1964}--\cite{Araki2009}.

In the papers \cite{Ruzzi2005}--\cite{Vasselli2012} the authors consider nets containing $C^*$-algebras of quantum observables for the case of curved spacetimes.  The net that is constructed by means of semigroup $C^*$-algebra generated by the path semigroup for  partially ordered set is studied in \cite{GLS2016}. The paper \cite{gumerov2018} contains results on limit
automorphisms  for inductive sequences of Toeplitz algebras which are closely related to the facts on the mappings of topological groups \cite{Gu05}, \cite{Gu05s}. In \cite{GLS2018} the authors deal with a net of $C^*$-algebras associated to a net over a partially ordered set consisting of Hilbert spaces.

In this paper we consider a covariant functor from a category associated with an arbitrary partially ordered set $K$ into the category of unital $C^*$-algebras and their unital $*$-homomorphisms.
That functor is also called an inductive system over $K$.
Using Zorn's lemma, the set $K$ can be represented as the union
of  the  family $\{K_i \}$ of its maximal upward directed subsets indexed by elements of a set $I$. We
consider a topology on the set $I$ generated by a base of neighbourhoods.  For every set $K_i$, $i\in I$, the original inductive system over $K$ yields naturally the inductive system of $C^*$-algebras over $K_i$ and its inductive limit. Using those inductive limits, we construct different types of $C^*$-algebras. In particular, for  neighbourhoods of the  topology on the set of indices we deal with the $C^*$-algebras which are the direct products of  limits for inductive systems over the sets $K_i$.

The present paper is devoted to the study of properties of  the above-mentioned topology and the $C^*$-algebras. We show that there exists a connection between topological and algebraic structures.

The paper consists of Introduction, three sections and Appendix.  The first section contains preliminaries.
In the second section we consider a  topology on the index set $I$ and study its properties. Examples of those topologies with different properties are given.
The third section deals with inductive limits. In this section a connection between topological and algebraic constructions is studied. Finally, Appendix contains the figures for the examples in the second section.

A part of the results in this paper was announced without proofs in \cite{GLG2018}.

\section{Preliminaries}\label{preliminaries}
\label{sec:1}

In what follows, we shall consider an arbitrary partially ordered set \ $(\,K, \, \leq \,)$ that is not necessarily directed.
The category associated to this set is denoted by the same letter \ $K$. We recall that the objects of this category are the elements of the set \ $K$, and, for any pair \ $a,b \in K$, the set of morphisms from \ $a$ \ to \ $b$ \ consists of the single element \ $(a,b)$ \ provided that \ $a\leq b$, and is the void set otherwise.

Further, we consider a covariant functor \ $\mathcal{F}$ \ from the category \ $K$ \ into the category of unital \ $C^*$-algebras and their unital \ $*$-homomorphisms. As was mentioned in Introduction, such a functor is called \emph{an inductive system }
in the category of \ $C^*$-algebras over the set \ $(K, \, \leq~)$. It may be given by a collection \ $(K,\{\mathfrak{A}_a\},\{\sigma_{ba}\})$ \ satisfying the properties from the definition of a functor.  We shall write \ $\mathcal{F}=(K,\{\mathfrak{A}_a\},\{\sigma_{ba}\})$.
 Here, $\{\mathfrak{A}_a\mid a\in K\}$ is a family of unital $C^*$-algebras. We also suppose that all morphisms $\sigma_{ba}:\mathfrak{A}_a\longrightarrow \mathfrak{A}_b$, where $a\leq b$, are embeddings of
$C^*$-algebras, i.\,e., unital injective $*$-homomorphisms. Recall that the equations
$\sigma_{ca}=\sigma_{cb}\circ\sigma_{ba}$ hold for all elements $a,b,c \in K$ satisfying the condition $a\leq b\leq
c$. Furthermore, for each element $a\in K$ the morphism $\sigma_{aa}$ is the identity mapping.

Considering the family of all upward directed subsets of the set \ $(\,K, \, \leq \,)$ and making use of \,  Zorn's lemma, one can easily prove

\begin{proposition}\label{Kdecomposition}
Let \ $(\,K, \, \leq \,)$ \ be a partially ordered set. Then the following equality holds:
\begin{equation}\label{KbigcupKi}
K=\bigcup\limits_{i\in I}K_i,
\end{equation}
where \ $\left \{ \, K_i \, | \, i\in I \, \right \}$ \ is the family of all maximal upward directed subsets of  \, $(K, \, \leq )$.

Moreover, for every $i\in I$ and $a\in K_i$ the set $\{b\in K \mid b\leq a\}$ is a subset of $ K_i$.

\end{proposition}


Now, for each  \ $i\in I$, we consider the inductive system \ $\mathcal{F}_i=(K_i,\{\mathfrak{A}_a\},\{\sigma_{ba}\})$ \
over the upward directed set \ $K_i$.

Throughout the paper, for a unital algebra $\mathfrak{A}$ its unit will be denoted by $\mathbb{I}_{\mathfrak{A}}$.

The simplest example of the inductive system $\mathcal{F}_i$ is that in which $\{\mathfrak{A}_a\mid a\in K_i\}$ is \emph{a net of $C^*$-subalgebras} of a given $C^*$-algebra $\mathfrak{A}$. By this, one means that each $\mathfrak{A}_a$ is a $C^*$-subalgebra containing the unit $\mathbb{I}_\mathfrak{A}$, $\mathfrak{A}_a\subset\mathfrak{A}_b$ and $\sigma_{ba}:\mathfrak{A}_a\longrightarrow\mathfrak{A}_b$ is the inclusion mapping whenever $a,b\in K_i$ and $a\leq b$. Given such a net $\mathcal{F}_i$, the norm closure of the union of all $\mathfrak{A}_a$ is itself a $C^*$-subalgebra of $\mathfrak{A}$.

 We recall the definition and  some facts concerning the inductive limits for inductive systems of $C^*$-algebras
(see, for example, \cite[Section 11.4]{KadisonRingrose}).

\emph{The inductive limit} of this system is a pair $(\mathfrak{A}^{K_i}, \{\sigma^{K_i}_a\})$ where
 \ $\mathfrak{A}^{K_i}$ is a $C^*$-algebra and $\{\sigma^{K_i}_a:\mathfrak{A}_a\rightarrow \mathfrak{A}^{K_i} \mid a\in
K_i\}$ is a family of unital injective $*$-homomorphisms such that the following two properties are fulfilled \cite[Proposition 11.4.1]{KadisonRingrose}:

1) For every pair elements $a,b\in K_i$ satisfying the condition $a\leq b$ the diagram
\[
\xymatrix{
  \mathfrak{A}_a \ar[rr]^{\sigma_{ba}} \ar[dr]_{\sigma^{K_i}_a}
                &  &    \mathfrak{A}_b \ar[dl]^{\sigma^{K_i}_b}    \\
                & \mathfrak{A}^{K_i}                 }
\]
is commutative, that is, the equality for mappings
\begin{equation}\label{sigma}
\sigma^{K_i}_a=\sigma^{K_i}_b\circ\sigma_{ba}.
\end{equation}
holds.
Moreover, we have the~following equality:
\begin{equation}\label{cup}
\mathfrak{A}^{K_i}=\overline{\mathop\bigcup\limits_{a\in
K_i}\sigma^{K_i}_a(\mathfrak{A}_a)},
\end{equation}
where the bar means the closure of the set with respect to the norm topology in the \ $C^*$-algebra \ $\mathfrak{A}^{K_i}$.

2) \emph{The universal property.} If $\mathfrak{B}$ is a $C^*$-algebra, $\psi_a:\mathfrak{A}_a\longrightarrow\mathfrak{B}$ is
an injective $*$-homomorphism for each $a\in K_i$, and conditions analogous to those in (\ref{sigma}) and (\ref{cup})
are satisfied, then there exists a $*$-isomorphism $\theta$ from $\mathfrak{A}^{K_i}$ onto $\mathfrak{B}$ such that the diagram
\begin{equation*}
\xymatrix{
                & \mathfrak{A}_a\ar[dr]^{\psi_a} \ar[dl]_{\sigma^{K_i}_a}            \\
 \mathfrak{A}^{K_i}  \ar[rr]^{\varphi} & &     \mathfrak{B}        }
\end{equation*}
is commutative for every $a\in K_i$, that is, the equality $\psi_a=\varphi\circ\sigma^{K_i}_a $ holds.

The inductive limit $(\mathfrak{A}^{K_i}, \{\sigma^{K_i}_a\})$  is denoted as follows:
$$
(\mathfrak{A}^{K_i}, \{\sigma^{K_i}_a\}):=\varinjlim\mathcal{F}_i.
$$
The $C^*$-algebra $\mathfrak{A}^{K_i}$ itself is often called the inductive limit.

It is well known that the inductive limit can always be constructed for an inductive system in the category of $C^*$-algebras. For the convenience of the reader, we recall briefly the components of that construction
(the details are contained in the proof of Proposition 11.4.1 in \cite{KadisonRingrose}). We shall use them in our proofs.

We consider the direct product of $C^*$-algebras
$$
\prod\limits_{a\in K_i}\mathfrak{A}_{a}:= \left\{ (F_a) \, \big| \, \|(F_a)\|=\sup_a\|F_a\|< + \infty \right\}.
$$
It is a $C^*$-algebra relative to the pointwise operations and the supremum norm. It has a norm-closed two-sided ideal $\Sigma$ consisting of all elements $(F_a)$ in the direct product for which the net $\{\|(F_a)\| \mid a\in K_i \}$ converges to 0. When $a\in K_i$ the $*$-homomorphism
$$\theta_a^{K_i}:\mathfrak{A}_{a}\longrightarrow \prod\limits_{b\in K_i}\mathfrak{A}_{b}$$
is defined at an element $A\in \mathfrak{A}_{a}$ as follows:
\begin{equation}\label{teta}
[\theta^{K_i}_a(A)](b)=
\begin{cases}
\sigma_{ba}(A), & \text{if  $a\leq b$}; \\
0, &\text{otherwise}.
\end{cases}
\end{equation}

The mapping $A\longrightarrow \theta^{K_i}_a(A)+\Sigma$ is a unital injective $*$-homomorphism $\sigma_a^{K_i}$ from $\mathfrak{A}_{a}$ into the quotient $C^*$-algebra $\prod\limits_{a\in K_i}\mathfrak{A}_{a}/\Sigma$. The family $\{\sigma_a^{K_i}(\mathfrak{A}_{a})\mid a\in K_i\}$ is a net of $C^*$-subalgebras of $\prod\limits_{a\in K_i}\mathfrak{A}_{a}/\Sigma$, and the norm closure
of $\bigcup\sigma_a^{K_i}(\mathfrak{A}_{a})$ is the inductive limit $\mathfrak{A}^{K_i}$ which is a $C^*$-subalgebra of $\prod\limits_{a\in K_i}\mathfrak{A}_{a}/\Sigma$.

It is worth noting that below we shall denote by the same symbols $\Sigma$ and $\theta_a^{K_i}$ the corresponding ideals and mappings for distinct inductive systems.

Finally, we construct the direct product for the inductive limits of the functors \ $\mathcal{F}_i$ \ denoted by
$$
\mathfrak{M}_{\mathcal{F}}:=\prod\limits_{i\in I}\mathfrak{A}^{K_i}= \left\{ (a_i) \, \big| \, \|(a_i)\|=\sup_i\|a_i\|< \infty \right\}.
$$

For additional results in the theory of $C^*$-algebras we refer the reader, for example,  \cite[Ch.\,4, \S\,7]{helemskii} and \cite{murphy}. Necessary facts from the theory of categories and functors are contained, for example, in \cite[ Ch.\,0, \S\,2]{helemskii} and \cite{BucurDeleanu1968}.

\section{Topology on the index set I}
Throughout this section, $K$ is an arbitrary partially ordered set that is not necessarily directed. By Proposition~\ref{Kdecomposition}, we have
equality (\ref{KbigcupKi}).

Now we endow the index set $I$ with a topology. To this end, for every element $a\in K$  we define the set \ $U_a=\{i\in I \ : \ a\in K_i\} $.
Obviously,  \ $U_a$ is a non-empty set.

Using the property of maximal upward directed sets $K_i$ from Proposition~\ref{Kdecomposition}, it is straightforward to prove

\begin{lemma}\label{fortopology}
If $a,b\in K$ such that $a\leq b$ then $U_b\subset U_a$.
\end{lemma}

We shall now show that the family of sets $\{U_a \mid a\in K\}$ generates a topology
on the index set $I$.

\begin{proposition}
The family $\{U_a \mid a\in K\}$ is a base for a topology on the set $I$.
\end{proposition}

\begin{proof}
To prove the assertion we use Proposition~1.2.1 from \cite{EngelkingGT}. According to it, we need to check two properties of a base.

Firstly, for any $U_a$ and $U_b$, $a,b\in K$, and every point $i\in U_a\bigcap U_b$ there exists
$U_c$, $c\in K$, such that $i\in U_c\subset U_a\bigcap U_b$. Indeed, since $i\in U_a\bigcap U_b$
the elements $a$ and $b$ belong to the maximal upward directed set $K_i$. Hence, there is an element $c\in K_i$ satisfying the conditions $a\leq c$ and $b \leq c$. By Lemma~\ref{fortopology}, the set $U_c$ satisfies the required property.

Secondly, it is clear that for every $i\in I$ there exists an element $a\in K$ such that $i\in U_a$.
\qed
\end{proof}

We denote by $\tau$ the topology generated by the base $\{U_a \mid a\in K\}$. Thus, $\tau$ is the family of all subsets of $I$ that are unions of subfamilies of  $\{U_a \mid a\in K\}$.

Below we prove the propositions describing properties of the topological space $(I,\tau)$ and
give several examples.

\begin{proposition}
The  topological space $(I,\tau)$ is a $T_1$-space.
\end{proposition}

\begin{proof}
Take any pair of distinct indices $i,j\in I$. The condition $i\neq j$ implies that $K_i\neq K_j$.
 Hence, we can take an element $a\in K_i\setminus
K_j$. Then we have the conditions $i\in U_a$ and $j\notin U_a$.

\qed
\end{proof}

\begin{corollary}
For every index $i\in I$ the equality $\mathop\bigcap\limits_{a\in K_i}U_a=\{i\}$ holds.
\end{corollary}

\begin{corollary}
For every index $i\in I$ the one-point set $\{i\}$ is closed.
\end{corollary}

The next example shows that $(I,\tau)$ may not be a Hausdorff space.

\begin{example}\label{tau:nonHausdorff}
We consider the set of points with integer coordinates
$$
K:=\left \{(x,y) \mid x\in  \{ -1; 0; 1 \}, y\in \mathbb{Z}\right \}.
$$
A partial order $\leq$ on the set $K$ is defined in the following way:
\[
(x_1,y_1)\leq(x_2,y_2) :=
\begin{cases}
x_1,x_2\in \{-1;1\},\, x_1=x_2, \, y_1\leq y_2\,; \\
x_1\in \{-1;1\},\, x_2=0, \, y_1< y_2.
\end{cases}
\]
It is straightforward to check that the pair  $(K,\leq)$  is a partially ordered set, which is not upward directed.

 One has the representation of $K$ as the union of maximal upward directed sets $K_i$ indexed by the set of all integers $\mathbb{Z}$ together with two symbols $-\infty$ and  $+\infty$, that is, $I=\mathbb{Z}\cup \{-\infty; +\infty \}$\,:
$$
K=\bigcup\limits_{i= -\infty}^{+\infty}K_i, \quad \text{where} \quad K_{-\infty}:=\left\{(-1,y)\mid y\in \mathbb{Z} \right\};\quad  K_{+\infty}:=\left\{(1,y)\mid y\in \mathbb{Z} \right\}
$$
$$
\text{and} \quad K_i:=\left\{(0,i)\right \} \bigcup \left\{(x,y) \mid  x\in \{-1;1\}, y<i, y\in \mathbb{Z} \right\} \quad \text{for each} \quad i\in \mathbb{Z}.
$$

A base $\{U_{(x,y)}\mid x\in \{-1;0;1\}, y\in \mathbb{Z} \}$ for the topology $\tau$ on the index set $I$ consists of the sets of three types, namely,
$$
U_{(-1,y)}:=\{-\infty\} \cup \{ i\in \mathbb{Z} \mid i>y\}; \
U_{(1,y)}:=\{+\infty\}\cup \{i\in \mathbb{Z} \mid i>y\}; \
U_{(0,y)}:=\{y\}.
$$
Since any two neighbourhoods of indices $-\infty$ and $+\infty$ have a non-empty intersection the topological space $(I,\tau)$ is not a Hausdorff space.
\end{example}

\begin{proposition}
For $a\in K$, the set $K^a$ is upward directed if and only if the neighbourhood $U_a$ consists of a single point.
\end{proposition}

\begin{proof}
\emph{Necessity.} Assume that $K^a$ is an upward directed set. By Zorn's Lemma, there exists an index $j\in I$ such that $K^a$ is contained in the maximal upward directed set $K_j$. Hence,
$a\in K_j$. Thus, by the definition of the neighbourhood $U_a$, we get the inclusion
\begin{equation}\label{isubsetUa}
\{j\}\subset U_a.
\end{equation}

Next, we shall show the reverse inclusion. In order to obtain a contradiction, we suppose that there exists an index $i\in I$ distinct from the index $j$ such that $i\in U_a$. Then, we have $a\in K_i$ and $K_i\neq K_j$.

In this case we can take an element
\begin{equation}\label{cinKisetminusKj}
c\in K_i \setminus K_j.
\end{equation}
Since $a,c \in K_i$ and $K_i$ is an upward directed set there exists an element $d\in K_i$ satisfying the conditions $a\leq d$ as well as $c\leq d$. The first condition implies that $d\in K^a \subset K_j$. Obviously, the latter together with the maximality property of  the upward directed set $K_j$ yields the inclusion
$c\in K_j$. This contradicts condition (\ref{cinKisetminusKj}). Therefore, we obtain the inclusion
that is reverse to (\ref{isubsetUa}). Thus, the equality
\begin{equation}\label{Uaj}
U_a = \{j\}
\end{equation}
holds, as required.

\emph{Sufficiency.} Let equality (\ref{Uaj}) be valid for some index $j\in I$. By the the definition of the neighbourhood $U_a$, the element $a$ belongs to the unique maximal upward directed set $K_j$.

We claim that the following inclusion holds:
\begin{equation}\label{KasubsetKj}
K^a\subset K_j.
\end{equation}
Indeed, take any $b\in K^a$. Then $a\leq b$ and, by Lemma~\ref{fortopology}, $U_b\subset U_a$. Therefore, $U_b=\{j\}$ and $b\in K_j$. Consequently,
we have inclusion (\ref{KasubsetKj}), as claimed.

To show that $K^a$ is an upward directed set we take two elements $b,c\in K^a$. By (\ref{KasubsetKj}), we have $b,c\in K_j$. Since $K_j$ is an upward directed set there exists
$d\in K_j$ such that  both the conditions $ b\leq d$ and $c\leq d$  hold. Obviously, we have
$d\in K^a$. Thus, the set $K^a$ is upward directed, as required. \qed

\end{proof}

Now we give  examples of  locally compact and discrete topological spaces.

\begin{example}\label{tau:locallycompact}
As the set $K$ we consider the lower half-plane without the axis $y=0$, that is,
$$K=\{(x,y) \ \mid \ x,y\in \mathbb{R}, y<0\}.$$

We define a partial order $\leq$ on \ $K$ \ as follows. Let us fix a positive number $a\in \mathbb{R}$\,. Then
we put

\[
(x_1,y_1)\leq(x_2,y_2) :=
\begin{cases}
x_1=x_2 \quad \text{and} \quad y_1= y_2; \\
y_2-y_1>a|x_2-x_1|.
\end{cases}
\]
It is easily verified that the pair $(K,\leq)$  is a partially ordered set. Moreover, it is worth noting  that
this set is not upward directed.

We have the representation of $K$ as the union of maximal upward directed sets $K_i$ indexed by the set of all real numbers, that is, $I=\mathbb{R}$\,:
$$
K=\bigcup\limits_{i\in
\mathbb{R}}K_i, \qquad \text{where} \quad K_i:=\left\{(x,y)\in K \mid -y>a|i-x| \right\}.
$$

Taking a point $(x_0,y_0)\in K$, one can easy see that
$$U_{(x_0,y_0)}=\Bigl\{i\in\mathbb{R} \ \mid \
x_0+\frac{y_0}{a}<i<x_0-\frac{y_0}{a}\Bigr\}.$$

Thus, in this example the topology $\tau$ coincides with the natural topology on the set $\mathbb{R}$ wich
is locally compact.

\end{example}

\begin{example}\label{tau:discrete}
Here, we take as $K$ the points with integer coordinates in the lower half-plane including the axis $y=0$, that is,
$$K=\{(n,m) \ : \ n,m\in \mathbb{Z}, m\leq 0\}.$$

We define a partial order $\leq$ on the set \ $K$ \ by the following rule:
$$(n_1,m_1)\leq (n_2,m_2) \qquad \text{\emph{if and only if}} \quad m_2-m_1\geq
n_2-n_1\geq 0.$$
The pair $(K,\leq)$  is a partially ordered set. It is not upward directed as well.

We have the representation of $K$ as the union of maximal upward directed sets $K_i$ indexed by the set of all integers, that is, $I=\mathbb{Z}$\,:
$$
K=\bigcup\limits_{i\in
\mathbb{Z}}K_i, \qquad \text{where} \quad K_i:=\left\{(n,m)\in K \mid  -m\geq i-n\geq 0 \right\}.
$$

For any point $(n_0,m_0)\in K$, in the space $I$ we have  the neighbourhood
$$U_{(n_0,m_0)}=\{i\in\mathbb{Z} \ : \
n_0\leq i\leq n_0-m_0\}.$$
Since the equality $U_{(n,0)}=\{n\}$ holds we see that every point in the space
 $I=\mathbb{Z}$ is isolated. Thus, we conclude that the topology $\tau$
is discrete.

\end{example}

\section{Inductive limits}

Let us consider an inductive system
$\mathcal{F}=(K,\{\mathfrak{A}_a\},\{\sigma_{ba}\})$, where
$\mathfrak{A}_a$ is an arbitrary unital $C^*$-algebra.
Then for each index $i\in I$ we can construct an inductive system $\mathcal{F}_i$ and its inductive limit $\mathfrak{A}^{K_i}$.

Further, we take any $a\in K$ and consider the direct product of $C^*$-algebras
 $$\mathfrak{B}_a:=\prod\limits_{i\in
U_a}\mathfrak{A}^{K_i}.$$

Recall, by Lemma~\ref{fortopology}, for every pair $a,b\in K$ satisfying the condition $a\leq b$,  we have the~inclusion $U_b\subseteq U_a$. Hence, the $*$-homomorphism $\tau_{ba}:\mathfrak{B}_a\longrightarrow\mathfrak{B}_b$ between $C^*$-algebras
given by the~rule
$$
\tau_{ba}(f)(j)=f(j),
$$
where $f\in \mathfrak{B}_a$ and $j\in U_b$, is well-defined.

Obviously, we have the equality $\tau_{ca}=\tau_{cb}\circ\tau_{ba}$ whenever $a,b,c\in K$ such that  the condition \ $a\leq b\leq c$ \ holds.

Thus we have constructed the inductive system of $C^*$-algebras $(K,\{\mathfrak{B}_a\},\{\tau_{ba}\})$.
Therefore, for each index $i\in I$ we can consider the inductive system $(K_i,\{\mathfrak{B}_a\},\{\tau_{ba}\})$ and its inductive limit
$$(\mathfrak{B}^{K_i}, \{\tau^{K_i}_a\} ):=\varinjlim(K_i,\{\mathfrak{B}_a\},\{\tau_{ba}\}).$$

For the direct product of these inductive limits we introduce the following notation:
$$\widehat{\mathfrak{M}}_{\mathcal{F}}:=\prod\limits_{i\in I}\mathfrak{B}^{K_i}.$$

Further, we prove the theorems that show the connections between  the structures of $C^*$-algebras $\mathfrak{A}^{K_i}$, $\mathfrak{B}^{K_i}$, $\mathfrak{M}_{\mathcal{F}}$, $\widehat{\mathfrak{M}}_{\mathcal{F}}$ and the properties of the topological space $(I,\tau)$.

\begin{theorem}\label{II} For every index $i$ in the set $I$  the algebra $\mathfrak{A}^{K_i}$ is isomorphic to a subalgebra of $\mathfrak{B}^{K_i}$.
\end{theorem}

\begin{proof}
 To construct an injective $*$-homomorphism from the algebra $\mathfrak{A}^{K_i}$ into the algebra $\mathfrak{B}^{K_i}$ we proceed as follows.

Take an arbitrary neighbourhood $U_a$ of the point $i$. For every index $j\in U_a$ we consider the inductive limit
$$
(\mathfrak{A}^{K_j}, \{\sigma_b^{K_j}\})=\varinjlim(K_j,\{\mathfrak{A}_b\},\{\sigma_{cb}\}).
$$

Using the family of injective $*$-homomorphisms $\{\sigma_a^{K_j}:\mathfrak{A}_a\longrightarrow
\mathfrak{A}^{K_j} \mid j\in U_a\}$, we define a $*$-homomorphism of $C^*$-algebras
$$\sigma_{a}^{U_a}:\mathfrak{A}_a\longrightarrow
\mathfrak{B}_a=\prod\limits_{j\in U_a}\mathfrak{A}^{K_j}
$$
by means of the formula \ $[\sigma_{a}^{U_a}(A)](j):=\sigma_a^{K_j}(A)$,
where \ $A\in \mathfrak{A}_a, j\in U_a$.
Note that the injectivity of  the $*$-homomorphisms $\sigma_a^{K_j}$ implies the injectivity of $\sigma_{a}^{U_a}$. Moreover, the following equalities hold:
\begin{equation}\label{normsigmaaUaAj}
\|[\sigma_{a}^{U_a}(A)](j) \|=\|\sigma_a^{K_j}(A)\|=\|A\|.
\end{equation}

Now we take two inductive systems $(K_i,\{\mathfrak{A}_a\},\{\sigma_{ba}\})$ and $(K_i,\{\mathfrak{B}_a\},\{\tau_{ba}\})$. For every pair of elements $a,b\in K_i$ satisfying the
condition $a\leq b$ we have the diagram

\begin{equation}\label{digrammathfrakAasigmaaUasigmaba}
\vcenter{
\xymatrix{
 \mathfrak{A}_a \ar[d]_{\sigma_{a}^{U_a}} \ar[r]^{\sigma_{ba}}
                &    \mathfrak{A}_b \ar[d]^{\sigma_{b}^{U_b}}  \\
  \mathfrak{B}_a  \ar[r]_{\tau_{ba}}
                &  \mathfrak{B}_b             }
}
\end{equation}
It is commutative, that is, we have the equality for the compositions of morphisms
\begin{equation}\label{commutativitysigmabUbsigmaba}
\sigma_{b}^{U_b}\circ\sigma_{ba}=\tau_{ba}\circ \sigma_{a}^{U_a}.
\end{equation}
 To show the validity of (\ref{commutativitysigmabUbsigmaba}) let us take any element $A\in \mathfrak{A}_a$. For every index $j\in U_b$ we have the following equalities:
$$
[\sigma_{b}^{U_b}\circ\sigma_{ba}(A)](j)=\sigma_b^{K_j}(\sigma_{ba}(A))=\sigma_a^{K_j}(A);
$$
$$
[\tau_{ba}\circ \sigma_{a}^{U_a}(A)](j)=\sigma_{a}^{U_a}(A)(j)=\sigma_a^{K_j}(A).
$$
Therefore, we obtain equation (\ref{commutativitysigmabUbsigmaba}).

 The commutativity of (\ref{digrammathfrakAasigmaaUasigmaba}) yields the commutativity of the diagram
\[
\xymatrix{
  \mathfrak{A}_a \ar[rr]^{\sigma_{ba}} \ar[dr]_{\tau^{K_i}_a\circ\sigma_a^{U_a}}
                &  &    \mathfrak{A}_b \ar[dl]^{\tau^{K_i}_b\circ\sigma_b^{U_b}}    \\
                & \mathfrak{B}^{K_i}                 }
\]
Indeed, we have the following equalities:
$$
\tau^{K_i}_a\circ\sigma_a^{U_a}=(\tau^{K_i}_b\circ \tau_{ba})\circ\sigma_a^{U_a}=(\tau^{K_i}_b\circ\sigma_b^{U_b})\circ\sigma_{ba}.
$$

We claim that for every $\textbf{a}\in K_i$ the $*$-homomorphism \ $\tau^{K_i}_a\circ\sigma_a^{U_a}$ \ is injective. To see that we take two arbitrary distinct elements $A_1$ and $A_2$ in the $C^*$-algebra $\mathfrak{A}_a$. For $l=1,2$ we have the representations
$$
\tau^{K_i}_a\circ\sigma_a^{U_a}(A_l)=\theta^{K_i}_a(\sigma_a^{U_a}(A_l))+\Sigma.
$$
We need to show that the condition
\begin{equation}\label{doesnotbelongtoSigma}
\theta^{K_i}_a(\sigma_a^{U_a}(A_1-A_2))\notin \Sigma
\end{equation}
holds, that is, the net $\{\|[\theta^{K_i}_a(\sigma_a^{U_a}(A_1-A_2))](x)\| \mid x\in K_i\}$ does not converge to $0$.

By  (\ref{teta}) for every element $x\in K_i$ one has
\begin{equation}\label{thetaKiasigmaaUaAlx}
[\theta^{K_i}_{a}(\sigma_a^{U_a}(A_l))](x)=
\begin{cases}
\tau_{xa}(\sigma_a^{U_a}(A_l)), & \text{if  $ a\leq x$}; \\
0, & \text{otherwise}.
\end{cases}
\end{equation}
In view of the commutativity of diagram (\ref{digrammathfrakAasigmaaUasigmaba}) with $x$ instead of $b$ we have
in formula (\ref{thetaKiasigmaaUaAlx}) the following:
\begin{equation}\label{tauxasigmaaUaAl}
\tau_{xa}(\sigma_a^{U_a}(A_l))=\sigma_x^{U_x}(\sigma_{xa}(A_l)).
\end{equation}
It follows from (\ref{thetaKiasigmaaUaAlx}), (\ref{tauxasigmaaUaAl}), (\ref{normsigmaaUaAj}) and the injectivity of the $*$-homomorphism $\sigma_{xa}$) that for every element $x\in K_i$ we get
\begin{equation}\label{normthetaKiasigmaaUaA1minusA2}
\|[\theta^{K_i}_a(\sigma_a^{U_a}(A_1-A_2))](x)\|=
\begin{cases}
\|A_1-A_2\|, & \text{if  $ a\leq x$}; \\
0, & \text{otherwise}.
\end{cases}
\end{equation}
Really, on the right-hand part of (\ref{normthetaKiasigmaaUaA1minusA2}) the first row is valid because we have the equalities:
\begin{equation*}
\begin{split}
\|\sigma_x^{U_x}(\sigma_{xa}(A_1-A_2))\|& =\sup\left\{\left\|[\sigma_x^{U_x}(\sigma_{xa}(A_1-A_2))](j)\right\| \Big |\, j\in U_x\right\} =\\
&=\sup\left\{\left\|\sigma_x^{K_j}(\sigma_{xa}(A_1-A_2))\right\| \Big |\, j\in U_x\right\} =\\
&=\left\|\sigma_{xa}(A_1-A_2))\right\| =\|A_1-A_2\|.
\end{split}
\end{equation*}

Hence, there exists $\varepsilon > 0$, for example, $\varepsilon=\frac{1}{2}\|A_1-A_2\|$, satisfying the following property: for every $y\in K_i$ there is an element $x\in K_i$ such that $y\leq x$, $a\leq x$ and
$$
\|[\theta^{K_i}_a(\sigma_a^{U_a}(A_1-A_2))](x)\|\geqslant \varepsilon.
$$
This means that condition (\ref{doesnotbelongtoSigma}) is satisfied, and the $*$-homomorphism \ $\tau^{K_i}_a\circ\sigma_a^{U_a}$ \ is an injection, as claimed.

Thus, by the universal property of inductive limits, we get the injective $*$-homomorphism of $C^*$-algebras \
$
\mathfrak{A}^{K_i}\longrightarrow \mathfrak{B}^{K_i},
$ as required.
\qed
\end{proof}

\begin{corollary} The $C^*$-algebra $\mathfrak{M}_{\mathcal{F}}$ is isomorphic to a subalgebra of $\widehat{\mathfrak{M}}_{\mathcal{F}}$.
\end{corollary}

\begin{theorem}~\label{theoremaboutisolated}
Let $i\in I$ be an isolated point in the topological space $(I,\tau)$.
Then the $C^*$-algebras
$\mathfrak{B}^{K_i}$ and $\mathfrak{A}^{K_i}$ are  isomorphic.
\end{theorem}

\begin{proof}
Since the one-point set $\{i\}$ is open in the topological space $(I,\tau)$ and the family
of sets $\{U_a \mid a\in K\}$ constitutes a base for the topology $\tau$ there exists an element
$a\in K_i$ such that we have the equality \ $U_a=\{i\}$. Fix that element $a\in K_i$.

By Lemma~\ref{fortopology},  the  neighbourhood  $U_b$ coincides with the one-point set $\{i\}$ whenever $b\in K$ and
$a\leq b$. Note that  such an element $b$ lies in the set $K_i$.

Further, we take the upward directed set
$$
K_i^a:=\{b\in K_i \mid a\leq b \},
$$
which is a cofinal subset in $K_i$.

Then we consider the inductive system $(K_i^a, \{\mathfrak{B}_b\},\{\tau_{cb}\})$ over the set $K_i^a$.
 For every element $b\in K_i^a$ we have the equality
 $\mathfrak{B}_b=\mathfrak{B}_a$. Note that the $C^*$-algebra $\mathfrak{B}_a$ consists of all the functions from the one-point set $\{i\}$ into the $C^*$-algebra $\mathfrak{A}^{K_i}$. Obviously, the $C^*$-algebras $\mathfrak{B}_a$ and $\mathfrak{A}^{K_i}$ are isomorphic. Moreover, in this system each bonding morphism $\tau_{cb}$  is
 the identity mapping. Thus, one has the isomorphism of $C^*$-algebras
 \begin{equation}~\label{somorpismKiaBbtaucbAKi}
 \varinjlim\limits(K_i^a, \{\mathfrak{B}_b\},\{\tau_{cb}\})\simeq\mathfrak{A}^{K_i}.
 \end{equation}

Further, we claim that there exists an isomorphism between $C^*$-algebras
\begin{equation}~\label{somorpismKiaBbtaucbBKi}
 \varinjlim\limits(K_i^a, \{\mathfrak{B}_b\},\{\tau_{cb}\})\simeq\mathfrak{B}^{K_i}.
 \end{equation}
The existence of such an isomorphism follows from the universal property for the inductive limits.

To show this, firstly, for every $b\in K_i^a$ we consider the $*$-homomorphism $$\tau^{K_i}_b:\mathfrak{B}_b\longrightarrow \mathfrak{B}^{K_i}.$$
It is injective. Indeed, we take an arbitrary non-zero element $f\in \mathfrak{B}_b$.
 Then we have
$$
\tau^{K_i}_b(f)=\theta^{K_i}_b(f) + \Sigma, 
$$
where the $*$-homomorphism $\theta^{K_i}_b:\mathfrak{B}_b\longrightarrow \prod\limits_{x\in K_i}\mathfrak{B}_x$ is given by the formula
\[
[\theta^{K_i}_b(f)](x)=
\begin{cases}
\tau_{xb}(f)=f, & \text{if  $b\leq x$}; \\
0, &\text{otherwise}.
\end{cases}
\]
Obviously, we have $\theta^{K_i}_b(f) \notin \Sigma$. Hence, $\tau^{K_i}_b$ is an injective $*$-homomorphism.

It is clear that  $\tau^{K_i}_b=\tau^{K_i}_c \circ \tau_{cb}$ whenever $b,c\in K^a_i$ and $b\leq c$.

Secondly, we prove the following equality:
\begin{equation}\label{eqBKiKai}
\mathfrak{B}^{K_i}=\overline{\bigcup\limits_{b\in K^a_i}\tau_b^{K_i}(\mathfrak{B}_b).}
\end{equation}
To this end, we recall that one has the equality
\begin{equation}\label{eqBKiKi}
\mathfrak{B}^{K_i}=\overline{\bigcup\limits_{x\in K_i}\tau_x^{K_i}(\mathfrak{B}_x).}
\end{equation}

Certainly, the right-hand part of (\ref{eqBKiKai}) is contained in the right-hand part of (\ref{eqBKiKi}). To obtain the reverse inclusion we fix an arbitrary element $x\in K_i$.
Since $K^a_i$ is cofinal in $K_i$ there is $b\in K^a_i$ such that $x\leq b$. The commutativity
of the diagram
\[
\xymatrix{
  \mathfrak{B}_x \ar[rr]^{\tau_{bx}} \ar[dr]_{\tau^{K_i}_x}
                &  &    \mathfrak{B}_b \ar[dl]^{\tau^{K_i}_b}    \\
                & \mathfrak{B}^{K_i}                 }
\]
yields the following equality for  sets
$$
\tau^{K_i}_b(\tau_{bx}(\mathfrak{B}_x))=\tau^{K_i}_x(\mathfrak{B}_x)
$$
which implies the desired inclusion for sets, namely,
$$
\tau^{K_i}_x(\mathfrak{B}_x)\subset \tau^{K_i}_b(\mathfrak{B}_b).
$$

Consequently, the right-hand parts of (\ref{eqBKiKai}) and (\ref{eqBKiKi}) coincide. Thus, we have proved equality (\ref{eqBKiKai}).

It follows from Proposition~11.4.1(ii) in \cite{KadisonRingrose} that there exists an isomorphism (\ref{somorpismKiaBbtaucbBKi}).

Finally,combining isomorphisms (\ref{somorpismKiaBbtaucbAKi}) and (\ref{somorpismKiaBbtaucbBKi}), we conclude that the  $C^*$-algebras
$\mathfrak{B}^{K_i}$ and $\mathfrak{A}^{K_i}$ are isomorphic, as required.
\qed
\end{proof}

The following statement is an immediate consequence of Theorem~\ref{theoremaboutisolated}.

\begin{corollary}
Let $(I,\tau)$ be a discrete  topological space.  Then the $C^*$-algebras
$\mathfrak{M}_{\mathcal{F}}$ and $\widehat{\mathfrak{M}}_{\mathcal{F}}$ are isomorphic.
\end{corollary}

\begin{theorem}\label{I} Let $i\in I$ be a non-isolated point with a countable neighbourhood base.
Then the algebra $\mathfrak{B}^{K_i}$ has a non-trivial center.
\end{theorem}

\begin{proof}
 It is clear that in the space $(I,\tau)$ one can construct a countable neighbourhood base
$\{U_{a_n} \mid a_n\in K_i, n\in \mathbb{N}\}$ at the point $i\in I$
satisfying the following conditions:
$$U_{a_1}\supset
U_{a_2}\supset U_{a_3}\supset \ldots ;$$
$$
a_1\leq a_2\leq
a_3\leq\ldots .
$$

Further, for each $n\in \mathbb{N}$ we consider the subset $W_n$ in $I$ given by 
$$W_n:=U_{a_n}\setminus
U_{a_{n+1}}.$$

Using these sets, we define the element $f$ in the subalgebra  $\prod\limits_{i\in
I}\mathbb{C} \mathbb{I}_{\mathfrak{A}^{K_i}}$ of the algebra $\mathfrak{M}_{\mathcal{F}}$. Namely,
for every $i\in I$ the value of the function $f$ at the point $i$ is defined as follows:
\begin{equation}\label{definitionfi}
f(i)=
\begin{cases}
0, & \text{if $i\notin U_{a_1}$;} \\
\mathbb{I}_{\mathfrak{A}^{K_i}}, & \text{if $i\in W_{2k-1}$, $k\in \mathbb{N}$;} \\
0, & \text{if $i\in W_{2k}$, $k\in \mathbb{N}$.}
\end{cases}
\end{equation}

Now for every element $a_n$ we consider two elements $f_{a_n}$ and $g_{a_n}$ in the subalgebra  $\prod\limits_{i\in
U_{a_n}}\mathbb{C} \mathbb{I}_{\mathfrak{A}^{K_i}}$ of the algebra $\mathfrak{B}_{a_n}$.
We put $f_{a_n}:=f|_{U_{a_n}}$, that is, the function $f_{a_n}$ is the restriction of the function $f$ to the neighbourhood $U_{a_n}$, and \ $g_{a_n}:=\mathbb{I}_{\mathfrak{B}_{a_n}}-f_{a_n}$.

Together with the functions $f_{a_n}$ and $g_{a_n}$ we define two elements $\tilde{f}$ and $\tilde{g}$ in the inductive limit $\mathfrak{B}^{K_i}$ by

\begin{equation}\label{definitiontildef}
 \tilde{f}:=\tau_{a_n}^{K_i}(f_{a_n}) \quad \text{and} \quad
\tilde{g}:=\tau_{a_n}^{K_i}(g_{a_n}).
\end{equation}
It is clear that these elements do not depend on the choice of the index $a_n$.

Obviously, we have the equality
$\tilde{f}\cdot\tilde{g}=0$ as well as
$\tilde{f}+\tilde{g}=\mathbb{I}_{\mathfrak{B}^{K_i}}$. Thus the elements $\tilde{f}$ and $\tilde{g}$ are non-trivial projections in the $C^*$-algebra $\mathfrak{B}^{K_i}.$

We claim that the elements $\tilde{f}$ and $\tilde{g}$ belong to the center of the algebra $\mathfrak{B}^{K_i}.$ To this end, we take an element $A=h+\Sigma$ of the $C^*$-algebra  $\mathfrak{B}^{K_i},$ where $h\in \prod\limits_{x\in K_i}\mathfrak{B}_x$. Let us show that the following equality holds:
\begin{equation}\label{fAAf}
\tilde{f}\cdot A = A \cdot\tilde{f}.
\end{equation}

Indeed, for the left-hand part of (\ref{fAAf}) we have the expression:
\begin{equation}\label{leftpart}
\tilde{f}\cdot A=\tau^{K_i}_{a_n}(f_{a_n})\cdot A =(\theta^{K_i}_{a_n}(f_{a_n})+\Sigma)\cdot(h+\Sigma)=\theta^{K_i}_{a_n}(f_{a_n})\cdot h+\Sigma.
\end{equation}

Analogously, for the right-hand part of (\ref{fAAf}) we get the representation
\begin{equation}\label{rightpart}
A\cdot \tilde{f}=h\cdot \theta^{K_i}_{a_n}(f_{a_n})+\Sigma.
\end{equation}

Further, for $x\in K_i$ we have
\begin{equation}\label{thetaKianfanhx}
[\theta^{K_i}_{a_n}(f_{a_n})\cdot h](x)=
\begin{cases}
\tau_{xa_n}(f_{a_n})\cdot h(x), & \text{if  $ a_n\leq x$}; \\
0, & \text{otherwise}.
\end{cases}
\end{equation}

For the function $\tau_{xa_n}(f_{a_n})\cdot h(x)$ from the algebra $\mathfrak{B}_x$ we take its value at a point $j\in U_x$:
\begin{equation}\label{tauxanfanhxj}
[\tau_{xa_n}(f_{a_n})\cdot h(x)](j)=[\tau_{xa_n}(f_{a_n})(j)]\cdot[h(x)(j)]=f_{a_n}(j)\cdot [h(x)(j)].
\end{equation}

Changing  the order of the factors in (\ref{thetaKianfanhx}) and  (\ref{tauxanfanhxj}), one gets the similar expressions for the summand $h\cdot \theta^{K_i}_{a_n}(f_{a_n})$ in the right-hand part of (\ref{rightpart}).

 By (\ref{definitionfi}) and the definition of the function $f_{a_n}$,  we obtain the equality
$$
f_{a_n}(j)\cdot [h(x)(j)]=[h(x)(j)]\cdot f_{a_n}(j).
$$

Therefore the elements in the right-hand parts of (\ref{leftpart}) and (\ref{rightpart}) are the same. Hence, equality (\ref{fAAf}) is proved. Similarly one can prove equality (\ref{fAAf}) with $\tilde{g}$ instead of $\tilde{f}$.

 It follows that the elements $\tilde{f}$ and $\tilde{g}$ belong to the center of the algebra $\mathfrak{B}^{K_i}$, as claimed. \qed
\end{proof}

As a consequence of Theorem~\ref{I} and the definition of the $C^*$-algebra $\widehat{\mathfrak{M}}_{\mathcal{F}}$ we have the following statement.
\begin{corollary} Let $(I,\tau)$ be a first-countable topological space without isolated points.
Then the $C^*$-algebra $\widehat{\mathfrak{M}}_{\mathcal{F}}$ has a non-trivial center.
\end{corollary}

\section*{Appendix: Figures for Examples}


 \emph{Example}~\ref{tau:nonHausdorff}.

\definecolor{uuuuuu}{rgb}{0.26666666666666666,0.26666666666666666,0.26666666666666666}
 \begin{center}
\begin{tikzpicture}[scale=0.5][line cap=round,line join=round,>=triangle 45,x=1cm,y=1cm]
\clip(-6.3,-6) rectangle (17,6);
\draw [->,line width=1pt] (-3,-4) -- (0,-2);
\draw [->,line width=1pt] (-3,-2) -- (0,0);
\draw [->,line width=1pt] (-3,0) -- (0,2);
\draw [->,line width=1pt] (-3,2) -- (0,4);
\draw [->,line width=1pt] (3,-4) -- (0,-2);
\draw [->,line width=1pt] (3,-2) -- (0,0);
\draw [->,line width=1pt] (3,0) -- (0,2);
\draw [->,line width=1pt] (3,2) -- (0,4);
\draw [->,line width=1pt] (-3,2) -- (-3,4);
\draw [->,line width=1pt] (3,2) -- (3,4);
\draw [->,line width=1pt] (-3,0) -- (-3,2);
\draw [->,line width=1pt] (3,0) -- (3,2);
\draw [->,line width=1pt] (-3,-2) -- (-3,0);
\draw [->,line width=1pt] (3,-2) -- (3,0);
\draw [->,line width=1pt] (-3,-4) -- (-3,-2);
\draw [->,line width=1pt] (3,-4) -- (3,-2);
\draw (-0.8,-2.5) node[anchor=north west] {$K_y$};
\draw (-0.8,-0.4) node[anchor=north west] {$K_{y+1}$};
\draw (-0.8,1.6) node[anchor=north west] {$K_{y+2}$};
\draw (-0.8,3.6) node[anchor=north west] {$K_{y+3}$};
\draw (-6,-1.6) node[anchor=north west] {$(-1, y)$};
\draw (3.5,-1.6) node[anchor=north west] {$(1, y)$};
\draw (-6,5) node[anchor=north west] {$K_{-\infty}$};
\draw (3.5,5) node[anchor=north west] {$K_{+\infty}$};
\draw (5.8,2) node[anchor=north west] {$U_{(1, y)}=\{+\infty\}\cup \{y+1, y+2, \dots\}$};
\draw (5.8,0.8) node[anchor=north west] {$U_{(-1, y)}=\{-\infty\}\cup\{y+1, y+2, \dots\}$};
\draw [line width=1pt] (-2.8,6)-- (-3.4,6);
\draw [line width=1pt] (-3.4,6)-- (-3.4,-6);
\draw [line width=1pt] (-3.4,-6)-- (-2.8,-6);
\draw [line width=1pt] (2.8,6)-- (3.4,6);
\draw [line width=1pt] (3.4,6)-- (3.4,-6);
\draw [line width=1pt] (3.4,-6)-- (2.8,-6);
\begin{scriptsize}
\draw [fill=black] (-3,-2) circle (1.5pt);
\draw [fill=black] (0,0) circle (2pt);
\draw [fill=black] (3,-2) circle (2pt);
\draw [fill=black] (-3,0) circle (2pt);
\draw [fill=black] (-3,2) circle (2pt);
\draw [fill=black] (-3,4) circle (2pt);
\draw [fill=black] (0,2) circle (2pt);
\draw [fill=black] (0,4) circle (2pt);
\draw [fill=black] (3,2) circle (2pt);
\draw [fill=black] (3,0) circle (2pt);
\draw [fill=black] (3,4) circle (2pt);
\draw [fill=black] (-3,-4) circle (2pt);
\draw [fill=black] (3,-4) circle (2pt);
\draw [fill=black] (0,-2) circle (2pt);
\draw [fill=black] (-3,-4.5) circle (1pt);
\draw [fill=black] (-3,-5) circle (1pt);
\draw [fill=black] (-3,-5.5) circle (1pt);
\draw [fill=black] (3,-4.5) circle (1pt);
\draw [fill=black] (3,-5) circle (1pt);
\draw [fill=black] (3,-5.5) circle (1pt);
\draw [fill=black] (-3,4.5) circle (1pt);
\draw [fill=black] (-3,5) circle (1pt);
\draw [fill=black] (-3,5.5) circle (1pt);
\draw [fill=black] (3,4.5) circle (1pt);
\draw [fill=black] (3,5) circle (1pt);
\draw [fill=black] (3,5.5) circle (1pt);
\draw [fill=black] (0,4.5) circle (1pt);
\draw [fill=black] (0,5) circle (1pt);
\draw [fill=black] (0,5.5) circle (1pt);
\draw [fill=black] (0,-4) circle (1.5pt);
\draw [fill=black] (0,-4.5) circle (1pt);
\draw [fill=black] (0,-5) circle (1pt);
\draw [fill=black] (0,-5.5) circle (1pt);
\end{scriptsize}
\end{tikzpicture}
\end{center}
\begin{center}
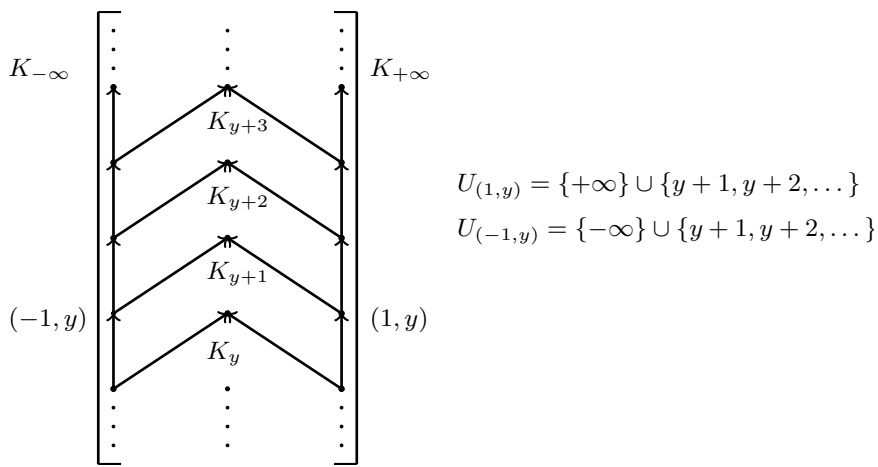
\textbf{Fig.1.} Non Hausdorff space
\end{center}

\emph{Example}~\ref{tau:locallycompact}.

\begin{center}
\begin{tikzpicture}[scale=0.7][line cap=round,line join=round,>=triangle 45,x=1cm,y=1cm]
\clip(-6.4,-6) rectangle (12,2.5);
\fill[line width=1.2pt,fill=black,pattern=dots,pattern color=black] (-4.925608929818391,-5.978900513088304) -- (-1.4279466666666663,0) -- (2.758824686291334,-5.978900513088304) -- cycle;
\draw [->,line width=1pt] (0,-6) -- (0,2.2);
\draw [->,line width=1pt] (-5,0) -- (9.5,0);
\draw [line width=0.8pt,dash pattern=on 4pt off 2.8pt] (0,2)-- (5.609407376598217,-6.010489662486432);
\draw [line width=0.8pt,dash pattern=on 4pt off 2.8pt] (6,2)-- (1,-6);
\draw [line width=1pt] (-5,0)-- (-7,0);
\draw (-5.6,-1.5) node[anchor=north west] {$(x_0; y_0)$};
\draw (3.3,-2.05) node[anchor=north west] {$(x_1; y_1)$};
\draw (-1.6,0.8) node[anchor=north west] {$i$};
\draw (-1.721550374273259,-2.860478053761196) node[anchor=north west] {$\mathbf{K_i}$};
\draw (4.9,-4.052719878728847) node[anchor=north west] {$y-y_1=-a (x-x_1)$};
\draw (5.1,0.8) node[anchor=north west] {$y-y_1=a (x-x_1)$};
\draw (-4.6,0.8) node[anchor=north west] {$U_{(x_0; y_0)}$};
\draw (2.4,-0.3) node[anchor=north west] {$(x_2; y_2)$};
\draw [line width=0.8pt,dash pattern=on 4pt off 2.8pt] (-4.925608929818391,-5.978900513088304)-- (-1.4279466666666663,0);
\draw [line width=0.8pt,dash pattern=on 4pt off 2.8pt] (-1.4279466666666663,0)-- (2.758824686291334,-5.978900513088304);
\draw [line width=0.8pt,dash pattern=on 4pt off 2.8pt] (-6.18393782068019,-5.978900513088304)-- (-2.4471249999999998,0);
\draw [line width=0.8pt,dash pattern=on 4pt off 2.8pt] (-4.957205428692007,0)-- (-0.7704340757340076,-5.978900513088304);
\draw [line width=2pt] (-4.9,0)-- (-2.5,0);
\draw (8.8,0) node[anchor=north west] {$x$};
\draw (-0.5293085493056073,2.15) node[anchor=north west] {$y$};
\draw (0,0) node[anchor=north west] {$O$};
\begin{scriptsize}
\draw [fill=black] (-1.4279466666666663,0) circle (2pt);
\draw [fill=black] (2.45,-0.8) circle (2pt);
\draw [fill=uuuuuu] (3.170361741259022,-2.5274212139855647) circle (2pt);
\draw [fill=uuuuuu] (-1.4279466666666663,0) circle (2pt);
\draw [color=uuuuuu] (-2.4471249999999998,0) circle (2.2pt);
\draw [color=uuuuuu] (-4.957205428692007,0) circle (2.2pt);
\draw [fill=uuuuuu] (-3.6308949355906432,-1.8940318969450298) circle (2pt);
\end{scriptsize}
\end{tikzpicture}
\end{center}
\begin{center}
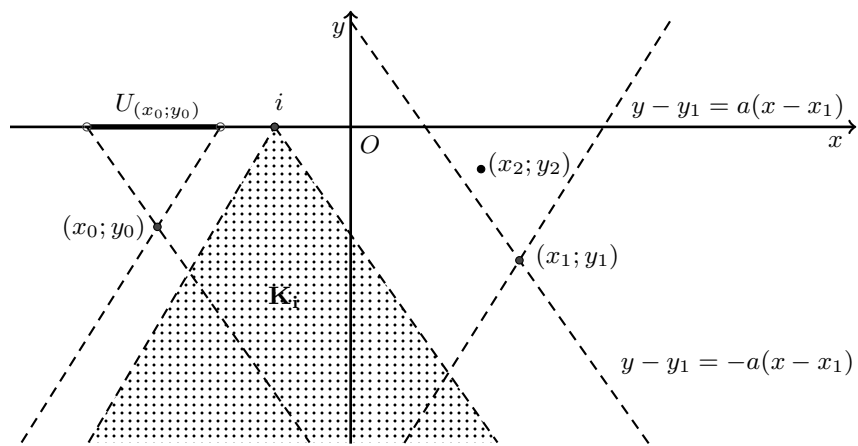
\textbf{Fig.2.} Locally compact space
\end{center}

\newpage

\emph{Example}~\ref{tau:discrete}.
\definecolor{uuuuuu}{rgb}{0.26666666666666666,0.26666666666666666,0.26666666666666666}
\begin{center}
\begin{tikzpicture}[scale=0.48][line cap=round,line join=round,>=triangle 45,x=1cm,y=1cm]
\clip(-12.2,-10.5) rectangle (20,4);
\fill[line width=1.2pt,fill=black,pattern=dots,pattern color=black] (-7.994634268609583,-9.994634268609582) -- (2,0) -- (2,-9.99910531138003) -- cycle;
\draw [->,line width=1pt] (-2,0) -- (12,0);
\draw (-8.9,-2.7) node[anchor=north west] {$(n_0; m_0)$};
\draw (4,-6) node[anchor=north west] {$(n_1; m_1)$};
\draw (-1.8,-6.4) node[anchor=north west] {$\mathbf{K_i}$};
\draw (-5.3,1.7) node[anchor=north west] {$U_{(n_0; m_0)}$};
\draw (5.5,-0.6) node[anchor=north west] {$(n_2; m_2)$};
\draw (10.8,0) node[anchor=north west] {$n$};
\draw (0.18354498231868652,1.6) node[anchor=north west] {$m$};
\draw (0,0) node[anchor=north west] {$O$};
\draw [line width=0.8pt] (4,-10)-- (4,0);
\draw [line width=0.8pt] (2,-10)-- (2,0);
\draw [line width=0.8pt] (-12,-10)-- (-2,0);
\draw [line width=0.8pt] (-8,-10)-- (2,0);
\draw [line width=0.8pt] (-6,0)-- (-6,-10);
\draw [line width=1pt] (-6,0)-- (-2,0);
\draw [line width=0.8pt] (-6,0.6)-- (-2,0.6);
\draw [line width=0.8pt] (-6,0.3)-- (-6,0.6);
\draw [line width=0.8pt] (-2,0.3)-- (-2,0.6);
\draw (1.7,1) node[anchor=north west] {$i$};
\draw [line width=1pt] (10,0)-- (0,-10);
\draw [->,line width=1.2pt] (0,-10) -- (0,1.6);
\draw [line width=1pt] (-12,0)-- (-6,0);
\begin{scriptsize}
\draw [fill=black] (8,0) circle (2pt);
\draw [fill=black] (-2,-10) circle (2pt);
\draw [fill=black] (4,-10) circle (2pt);
\draw [fill=uuuuuu] (2,0) circle (3pt);
\draw [fill=black] (-2,0) circle (3pt);
\draw [fill=black] (2,-10) circle (2pt);
\draw [fill=uuuuuu] (4,-4) circle (2pt);
\draw [fill=uuuuuu] (-12,-10) circle (2pt);
\draw [fill=uuuuuu] (-8,-10) circle (2pt);
\draw [fill=uuuuuu] (-6,-10) circle (2pt);
\draw [fill=black] (-6,0) circle (3pt);
\draw [fill=uuuuuu] (-6,-4) circle (3pt);
\draw [fill=black] (-12,-2) circle (2pt);
\draw [fill=black] (-10,-2) circle (2pt);
\draw [fill=black] (-8,-2) circle (2pt);
\draw [fill=black] (-6,-2) circle (2pt);
\draw [fill=black] (-4,-2) circle (2pt);
\draw [fill=black] (-2,-2) circle (2pt);
\draw [fill=black] (-12,0) circle (2pt);
\draw [fill=black] (-10,0) circle (2pt);
\draw [fill=black] (-8,0) circle (2pt);
\draw [fill=black] (-4,0) circle (3pt);
\draw [fill=black] (-0,0) circle (2pt);
\draw [fill=black] (-12,-4) circle (2pt);
\draw [fill=black] (-10,-4) circle (2pt);
\draw [fill=black] (-12,-6) circle (2pt);
\draw [fill=black] (-10,-6) circle (2pt);
\draw [fill=black] (-12,-8) circle (2pt);
\draw [fill=black] (-10,-8) circle (2pt);
\draw [fill=black] (-8,-4) circle (2pt);
\draw [fill=black] (-8,-6) circle (2pt);
\draw [fill=black] (-6,-6) circle (2pt);
\draw [fill=black] (-8,-8) circle (2pt);
\draw [fill=black] (-6,-8) circle (2pt);
\draw [fill=black] (-4,-6) circle (2pt);
\draw [fill=black] (-4,-4) circle (2pt);
\draw [fill=black] (-2,-4) circle (2pt);
\draw [fill=black] (-2,-6) circle (2pt);
\draw [fill=black] (-2,-8) circle (2pt);
\draw [fill=black] (-4,-8) circle (2pt);
\draw [fill=black] (-4,-10) circle (2pt);
\draw [fill=black] (0,-10) circle (2pt);
\draw [fill=black] (0,-2) circle (2pt);
\draw [fill=black] (0,-4) circle (2pt);
\draw [fill=black] (0,-6) circle (2pt);
\draw [fill=black] (0,-8) circle (2pt);
\draw [fill=black] (2,-2) circle (2pt);
\draw [fill=black] (2,-4) circle (2pt);
\draw [fill=black] (2,-6) circle (2pt);
\draw [fill=black] (2,-8) circle (2pt);
\draw [fill=black] (4,-8) circle (2pt);
\draw [fill=black] (4,-6) circle (3pt);
\draw [fill=black] (4,-2) circle (2pt);
\draw [fill=black] (6,0) circle (2pt);
\draw [fill=black] (6,-2) circle (3pt);
\draw [fill=black] (6,-4) circle (2pt);
\draw [fill=black] (6,-6) circle (2pt);
\draw [fill=black] (6,-8) circle (2pt);
\draw [fill=black] (8,-8) circle (2pt);
\draw [fill=black] (8,-6) circle (2pt);
\draw [fill=black] (8,-4) circle (2pt);
\draw [fill=black] (8,-2) circle (2pt);
\draw [fill=black] (10,-2) circle (2pt);
\draw [fill=black] (10,-4) circle (2pt);
\draw [fill=black] (10,-6) circle (2pt);
\draw [fill=black] (10,-8) circle (2pt);
\draw [fill=black] (10,-10) circle (2pt);
\draw [fill=uuuuuu] (10,0) circle (2pt);
\draw [fill=black] (-10,-10) circle (2pt);
\draw [fill=black] (6,-10) circle (2pt);
\draw [fill=black] (8,-10) circle (2pt);
\draw [fill=black] (4,0) circle (2pt);
\end{scriptsize}
\end{tikzpicture}
\end{center}

\begin{center}
\textbf{Fig.3.} Discrete space
\end{center}
%
%



\begin{thebibliography}{}
%
%
\bibitem{HaagKastler1964}
Haag R. and Kastler D. An algebraic approach to quantum field theory, J. Math.
Phys. 5, 848 (1964).

\bibitem{Haag1996}
Haag R. Local quantum physics: fields, particles, algebras, 390 p. Springer Texts and Monographs in Physics, 2nd. rev. and enlarged ed. (1996).

\bibitem{Araki2009}
Araki H. Mathematical Theory of Quantum Fields, 236 p. Oxford University Press, Oxford (2009).

\bibitem{Roberts2000}
Roberts, J.E., More lectures on algebraic quantum field theory. In: Noncommutative Geometry: Lectures
Given at the C.I.M.E. Summer School Held in Martina Franca, Italy. Lecture Notes in Mathematics,
vol. 1831, pp. 263--342. Springer, Berlin (2004).

\bibitem{Horuzhy}
Horuzhy S.S. Introduction to algebraic quantum field theory.  Mathematics and its Applications
(Soviet Series), 301 p. Kluwer, Dordrecht, (1990).
\bibitem{BDFY2015}
Brunetti R., Dappiaggi C,. Fredenhagen K., Yngvason J. (Eds.), Advances in Algebraic Quantum Field Theory, Mathematical Physics Studies, 453 p. Springer International Publishing (2015).

\bibitem{Ruzzi2005} Ruzzi G. Homotopy of posets, net-cohomology and superselection
sectors in globally hyperbolic space-times, Rev. Math. Phys.
\textbf{17} (9), 1021--1070 (2005).

\bibitem{RuzziVasselli2012} Ruzzi G., Vasselli E. A new light on nets of $C^*$-algebras
and their representations, Comm. Math. Phys. \textbf{312} (3),
655--694 (2012).

\bibitem{Vasselli2012} Vasselli E. Presheaves of symmetric tensor categories and nets
of $C^*$-algebras, J. Noncommut. Geometry \textbf{9} (1), 121--159 (2015).

\bibitem{GLS2016} Grigoryan S., Grigoryan T., Lipacheva E., Sitdikov A. $C^*$-algebra generated by the paths semigroup, Lobachevskii J. Math. \textbf{37} (6), 740--748 (2016).

\bibitem{gumerov2018}
 Gumerov R. N. Limit Automorphisms of $C^*$-algebras generated by Isometric Representations for Semigroups of Rationals,  Siberian Math. J. \textbf{59} (1), 73--84 (2018).

\bibitem{GLS2018}
Grigoryan S., Lipacheva E., Sitdikov A. Nets of graded $C^*$-algebras over partially ordered sets, Algebra and Analysis,  \textbf{30} (6), (2018) (appear) (in Russian).

 \bibitem{Gu05} Gumerov R.N. On finite-sheeted covering mappings onto solenoids,
 Proc. Amer. Math. Soc. \textbf{133} (9), 2771--2778 (2005).


 \bibitem{Gu05s} Gumerov R.N. Weierstrass Polynomials and Coverings of Compact Groups,
 Sib. Math. J. \textbf{ 54} (2), 243–-246 (2013).



\bibitem{GLG2018}
Gumerov R.N., Lipacheva E.V., Grigoryan T.A. On Inductive Limits for Systems of \ $C^*$~-~algebras,
Russian Math.(Izvestiya VUZ. Matematika), \textbf{62} (7),  68--73 (2018).
\bibitem{KadisonRingrose} Kadison R.V., Ringrose J.R. Fundamentals of the theory of operator algebras, Volume II, Advanced theory, 399--1074 pp., Academic Press. Inc., London (1986).





\bibitem{helemskii}
Helemskii A.Ya. Banach and locally convex algebras, 446 p. Oxford
Science Publications. The Clarendon Press Oxford University Press, New York (1993).

\bibitem {murphy}
Murphy  G.J. $C^*$-algebras and operator theory, 296 p. Academic
Press, New York (1990).

\bibitem{BucurDeleanu1968}
Bucur I., Deleanu A., with the collaboration of Hilton P.J.
 Introduction to the Theory of Categories and Functors,
 Pure and Appl. Math., V. XIX, 224 p. Wiley - Interscience Publ.,
London - New York - Sydney (1968).









\bibitem{EngelkingGT} Engelking R. General topology, 626 p. PWN -- Polish Shientific publishers, Warszawa (1977).


\end{thebibliography}


\end{document}